\documentclass[a4paper]{article}

\usepackage[utf8]{inputenc}
\usepackage[T1]{fontenc}
\usepackage{times}

\usepackage{xspace}

\usepackage{longtable}
\usepackage{amsmath,amsthm,amssymb,amsthm}
\usepackage{enumerate}

\usepackage{graphicx}
\usepackage{amssymb,amsthm}

\usepackage{times}

\usepackage{thmbox}

\newtheorem{theorem}{Theorem}[section]
\newtheorem{proposition}[theorem]{Proposition}

\newtheorem{corollary}[theorem]{Corollary}

\newcommand{\ora}[1]{\overrightarrow{#1}}
\DeclareMathOperator{\dic}{\ora \chi}

\newcommand{\R}{\mathbb{R}}

\title{Dichromatic Number and Cycle Inversions}
\author{Pierre Charbit and Stéphan Thomassé}
\date{3 Janvier 2020}                                           

\begin{document}
\maketitle

\begin{abstract}
The results of this note were stated in the first author PhD manuscript in 2006  but never published. The writing of the proof given there was slightly careless and the proof itself scattered accross the document, the goal of this note is to give a short and clear proof using Farkas Lemma. The first theorem is a characterization of the acyclic chromatic number of a digraph in terms of cyclic ordering. Using this theorem we prove that for any digraph, one can sequentially reverse the orientations of the arcs of a family of directed cycles so that the resulting digraph has acyclic chromatic number at most $2$.
\end{abstract}

\section{Introduction}
In this paper all directed graphs (\textit{digraphs} in short) are simple, i.e. contain no loop and no multi-arc. 
Given a digraph $D$, we denote by $V(D)$ its set of vertices and by $A(D)$ its set of arcs. 
%
An {\em acyclic colouring} of a digraph is an assignment of colours to the vertices such that each colour class induces an {\em acyclic subdigraph}, that is a subdigraph containing no directed cycle.  The \textit{acyclic chromatic number}, or simply \emph{dichromatic number}, of a digraph $D$, is defined to be the smallest number of colours required for an acyclic colouring of $D$ and is denoted $\dic(D)$. This notion was first introduced in 1982 by Neumann-Lara \cite{NL82} and  seems to be the natural generalization for digraphs of the usual chromatic number. 

The main result is to give a characterization of the existence of a proper $k$-acyclic colouring of a digraph. From it we deduce a theorem about circuit inversions, which was extended in 2020 for infinite digraphs by Ellis in Soukup in \cite{infinite}.

\section{The main result and its proof}

\begin{theorem}\label{thm:main}
A digraph $D$ can be partitioned into $k$ acyclic subgraphs (i.e. $\dic(D)\leq k$) if and only if there is an ordering $\sigma$ of the vertices of $D$ such that for any circuit $C$, the number of arcs of $C$ going forward in $\sigma$ is at least $|C|/k$.
\end{theorem}

Denote by $n=|V(D)|$ and $m=|A(D)|$ the number of vertices and arcs of $D$. In order to prove the result we consider the classical incidence matrix of the digraph which is a $n\times m $ matrix that we will denote $M_D$, where each column corresponds to an arc $a=xy$ of $D$ and contains only two non zero entries : an entry $(-1)$ for the row corresponding to the tail $x$ and entry $1$ for one corresponding to the head $y$. For a vector $z\in \R^n$ seen as a matrix with one column, ${}^{t}z$ denotes the row vector obtained by transposition.

\ 
\begin{proposition}
Let $D=(V,A)$ be a digraph with $n$ vertices and $m$ arcs, and denote by $M=M_D$ its incidence matrix. Let $p\in R^{m}$  be a weighting of the arcs, seen as a column vector. The three following propositions are equivalent : 
\begin{enumerate}
\item $\forall C \mbox{ directed circuit of } D,  \sum_{a\in A(C)}p(a) \geq 0$
\item $\forall c\in \R_{+}^{m}\ \ M.c=0 \Rightarrow \ {}^{t}p.c\, \geq 0 $
\item $\exists z \in \R^{n},\ {}^{t}z.M\leq p$
\end{enumerate}
\end{proposition}

\begin{proof}
The first two are easily seen to be equivalent. Each positive vector in the kernel is a weighting such that for every vertex the sum of the weights of the out arcs is equal to the sum of the weights of the arcs entering it. There must be a directed circuit of positive arcs. One can remove this weighted circuit and iterate.

The equivalence between (2) and (3) is one version of the classical Farkas Lemma (see for example \cite{MG}, page 92)
\end{proof}

\begin{proof}[of Theorem \ref{thm:main}]
The implication from left to right is easy : we first put all vertices of the first color class in a way that all arcs go forward, then do similarly for a second color class, and so on. The only backward arcs go from a color class to another one placed before, so there cannot be more than $k-1$ successive back arcs in any directed path, which proves the desired inequality for any circuit.

Let us prove now the converse direction and let $\sigma$ be an ordering of the vertices satisfying the conditions of the theorem. We define the following weighting of the arcs : 
$$p(a)=forw_{\sigma}(a)-1/k$$ where $forw_{\sigma}(a)=1$ if $a$ forward arc in $\sigma$ and $0$ otherwise. 
 
The hypothesis on $\sigma$ easily implies that every directed circuit has non negative weight for $p$. So by Proposition 1, there exists $z \in \R^{V(D)}$ such that 
$$\forall a=xy \in A, z(y)-z(x)\leq forw_{\sigma}(a)-1/k,$$

For $i\in \mathbb Z$ and $t\in \{0,\ldots,k-1\}$, let $I_{i,t}= [i+\frac t k ,i+ \frac {t+1} k[$ et $D_{t}=z^{-1}(\cup_{i\in \mathbb Z}I_{i,t})$.

Now we claim that $D_{t}$ induces an acyclic digraph. Indeed :
\begin{itemize}
\item There is no arc from $I_{i,t}$ to $I_{j,t}$ with $i<j$, since $ z(y)-z(x)\leq (1-1/k)$ for every arc.
\item If $a=xy$ is an arc with $z(x)$ and $z(y)$ both in $I_{i,t}$ then  necessarily  $z(y)-z(x)>-1/k$  so necessarily $forw_{\sigma}(a)=1$. Therefore the digraph induced by $I(i,t)$ must be acyclic.
\end{itemize}
This concludes the proof
\end{proof}

\section{A corollary about cycle inversions}
\begin{corollary}
For any digraph $D$, it is possible to reverse some of its circuits so that the resulting digraph $D'$ has acylic chromatic number at most $2$.
\end{corollary}
{\noindent \bf Preuve :}\\
Pick an arbitrary order. As long as there is one circuit whose number of backward arcs is larger than its number of forward arcs, reverse it. By doing this the total number of forward arcs always increase and since it cannot increase infinitely, one gets the desired ordering.
\qed

\end{document}